\documentclass[12pt]{article}

\usepackage{enumerate}
\usepackage{amsmath}
\usepackage{amssymb,latexsym}
\usepackage{amsthm}
\usepackage{color}
\usepackage{listings}
\usepackage{graphicx}
\usepackage{hyperref}
\usepackage{amscd}
\usepackage{graphicx}
\usepackage{array}

\DeclareMathOperator{\Tr}{Tr}
\DeclareMathOperator{\N}{N}

\title{MRD-codes arising from the trinomial $x^q+x^{q^3}+cx^{q^5}\in\F_{q^6}[x]$}
\author{Giuseppe Marino, Maria Montanucci and Ferdinando Zullo}

\newcommand{\cC}{{\mathcal C}}

\newcommand{\cL}{{\mathcal L}}

\newcommand{\F}{{\mathbb F}}

\newcommand{\la}{\langle}
\newcommand{\ra}{\rangle}

\newcommand{\D}{\mathcal{D}}

\newtheorem{theorem}{Theorem}[section]

\newtheorem{proposition}[theorem]{Proposition}
\newtheorem{result}[theorem]{Result}

\DeclareMathOperator{\PG}{\mathrm{PG}}

\DeclareMathOperator{\GL}{{GL}}

\DeclareMathOperator{\PGaL}{P\Gamma L}

\begin{document}

\maketitle

\begin{abstract}
In \cite{CsMZ2018}, the existence of $\F_q$-linear MRD-codes of $\F_q^{6\times 6}$, with dimension $12$, minimum distance 5 and left idealiser isomorphic to $\F_{q^6}$, defined by a trinomial of $\F_{q^6}[x]$, when $q$ is odd and $q\equiv 0,\pm 1\pmod 5$, has been proved.  In this paper we show that this family produces $\F_q$-linear MRD-codes of $\F_q^{6\times 6}$, with the same properties, also in the remaining $q$ odd cases, but not in the $q$ even case. These MRD-codes are  not equivalent to the previously known MRD-codes. We also prove that the corresponding maximum scattered $\F_q$-linear sets of $\PG(1,q^6)$ are not $\PGaL(2,q^6)$-equivalent to any previously known linear set.

\end{abstract}



\section{Introduction and preliminary results}
Let ${\rm End}_{\F_q}(\F_{q^n}):={\rm Hom}_q(\F_{q^n},\F_{q^n})$ be the set of all $\F_q$-linear maps of $\F_{q^n}$ in itself. It is well-known that each element $f$ of  ${\rm End}_{\F_q}(\F_{q^n})$ can be represented in a unique way as a $q$-polynomial over $\F_{q^n}$ of degree less than or equal to $q^{n-1}$, that is $f(x)=\sum_{i=0}^{n-1} a_i x^{q^{i}}$, with coefficients in $\F_{q^n}$. Such polynomials are also called \emph{linearized}. The set of $q$-polynomials over $\F_{q^n}$, say $\cL_{n,q}$, considered modulo $(x^{q^n}-x)$, and endowed with the addition and composition of polynomials in $\F_{q^n}$ and scalar multiplication by elements in $\F_q$, forms an $\F_q$-subalgebra of the algebra of $\F_q$-linear transformations of $\F_{q^n}$. Hence, we
can define the \emph{kernel} of $f$ as the kernel of the corresponding $\F_q$-linear transformation of $\F_{q^n}$, which is the same as the set of roots of $f$ in $\F_{q^n}$; and the \emph{rank} of $f$ as the rank of the corresponding $\F_q$-linear transformation of $\F_{q^n}$.
For $f\in \cL_{n,q}$ with $\deg f = q^{k}$,  we call $k$ the \emph{$q$-degree} of $f$ and we denote it by $\deg_{q} f$.
It is clear that in this case the kernel of $f$ has dimension at most $k$ and the rank of $f$ is at least $n-k$.

In \cite{CsMPZ2018}, the $q$-polynomials $f$ such that $\dim_{\F_q} \ker\,f = \deg_{q} f$ are called $q$-polynomials \emph{with maximum kernel}. Also in \cite[Theorem 1.2]{CsMPZ2018}, sufficient and necessary conditions for the coefficients of a $q$-polynomial $f$ over $\F_{q^n}$ ensuring $f$ has maximum kernel are given (see also \cite{McGuireSheekey}).

The set $\F_q^{m\times n}$ of all $m\times n$ matrices over $\F_q$ is a rank metric $\F_q$-space with the \emph{rank metric} or the \emph{rank distance} defined by
\[d(A,B)=\mathrm{rank}(A-B),\]
for any $A,B\in \F_q^{m\times n}$.
A subset $\cC\subseteq \F_q^{m\times n}$ with respect to the rank metric is usually called a \emph{rank-metric code} or a \emph{rank-distance code} (or \emph{RD-code} for short). When $\cC$ contains at least two elements, the \emph{minimum distance} of $\cC$ is given by
\[d(\cC)=\min_{A,B\in \cC, A\neq B} \{d(A,B)\}.\]
When $\cC$ is an $\F_q$-linear subspace of $\F_q^{m\times n}$, we say that $\cC$ is an $\F_q$-linear code and its dimension $\dim_{\F_q}(\cC)$ is defined to be the dimension of $\cC$ as a subspace over $\F_q$. For any $\cC\subseteq \F_q^{m\times n}$ with $d(\cC)=d$, it is well-known that
\[\#\cC\le q^{\max\{m,n\}(\min\{m,n\}-d+1)},\]
which is a Singleton like bound for the rank metric (\cite{Delsarte}).  When equality holds, we call $\cC$ a \emph{maximum rank-distance} (MRD for short) code.

\medskip

In this paper we only consider $\F_q$-linear RD and MRD-codes with $m=n$.

\medskip

Two $\F_q$-linear rank-distance codes $\cC_1$ and $\cC_2$ in $\F_q^{n\times n}$ are \emph{equivalent} if there exist $A,B\in {\rm GL}(n,q)$ and $\rho\in{\rm Aut(\F_q)}$ such that $\cC_2=\{AM^\rho B\colon M\in\cC_1\}$.
%

In general, it is a difficult task to tell whether two given rank-distance codes are equivalent or not. The idealizers of an RD-code are useful invariants which may help us to distinguish them (see \cite{liebhold_automorphism_2016,LTZ2,GiuZ}).
Given an $\F_q$-linear rank-distance code $\cC\subseteq \F_q^{n\times n}$, following \cite{liebhold_automorphism_2016} its \emph{left} and \emph{right idealisers} are defined as
\[L(\cC) =\{M\in\F_q^{n\times n} \colon MC\in \cC \text{ for all }C\in \cC  \},\]
and
\[R(\cC) =\{M\in\F_q^{n\times n} \colon CM\in \cC \text{ for all }C\in \cC  \},\]
respectively.

The \emph{adjoint} of an $\F_q$-linear RD-code $\cC\subseteq \F_q^{n\times n}$ is the $\F_q$-linear code
\[ \cC^\top:=\{C^T \in \F_q^{n\times n} \colon C \in \cC\},\]
where $(\,.\,)^T$ denotes the transpose operation.  Note that the adjoint operation also preserves rank distance, implying that an $\F_q$-linear RD-code and its adjoint have the same minimum distance. Also $L(\cC)=R(\cC^T)$ and $R(\cC)=L(\cC)$ (\cite[Prop. 4.2]{LTZ2}).

The \emph{Delsarte dual code} of an $\F_q$-linear code $\cC\subseteq \F_q^{n\times n}$ is
\[\cC^\perp :=\{M\in \F_q^{n\times n} \colon \Tr(MN^T)=0 \text{ for all } N\in \cC  \},\] where $(\,.\,)^T$ denotes the transpose operation.
If $\cC$ is a linear MRD-code then $\cC^\perp$ is also a linear MRD-code as it was proved by Delsarte \cite{Delsarte}. Also from \cite{Delsarte}, if $\cC$ is an $\F_q$-linear code $\cC\subseteq \F_q^{n\times n}$ with dimension $k$ and minimum distance $d$, then $\cC^\perp$ has dimension $n(n-k)$ and minimum distance $k+1$.

It is well-known that two linear rank-distance codes are equivalent if and only if their adjoint codes (or their Delsarte duals) are equivalent.


\bigskip

Two MRD-codes in $\F_q^{n\times n}$ with minimum distance $n$ are equivalent if and only if the corresponding semifields are isotopic \cite[Theorem 7]{lavrauw_semifields_2011}.
In contrast, it appears to be difficult to obtain inequivalent MRD-codes in $\F_q^{n\times n}$ with minimum distance strictly less than $n$.
So far, the known inequivalent MRD-codes in $\F_q^{n\times n}$ of minimum distance strictly less than $n$, can be divided into two types.
\begin{enumerate}
	\item The first type of constructions consists of MRD-codes of minimum distance $d$ for arbitrary $2\le d\le n$.
	\begin{itemize}
		\item The first construction of MRD-codes was given by Delsarte~\cite{Delsarte} and rediscovered independently by Gabidulin~\cite{Gabidulin}.  This construction was generalized by Kshevetskiy and Gabidulin~\cite{kshevetskiy_new_2005} with the nowadays commonly called (generalized) Gabidulin codes. In 2016, Sheekey~\cite{Sh} found the so-called (generalized) twisted Gabidulin codes. They can be generalized into additive MRD-codes \cite{otal_additive_2016}. Very recently, by using skew polynomial rings Sheekey~\cite{sheekey_new_arxiv} proved that they can be further generalized into a quite large family and all the MRD-codes mentioned above can be obtained in this way.
		\item The non-additive family constructed by Otal and \"Ozbudak in \cite{otal_additive_2018}.
		\item The family appeared in \cite{trombetti_new_2017}.
	\end{itemize}
	\item The second type of constructions provides us MRD-codes of minimum distance $d=n-1$.
	\begin{itemize}
		\item Non-linear MRD-codes by Cossidente, the second author and Pavese \cite{cossidente_non-linear_2016} which were later generalized by Durante and Siciliano \cite{durante_nonlinear_MRD_2017}.
		\item Linear MRD-codes associated with maximum scattered linear sets of $\PG(1,q^6)$ and $\PG(1,q^8)$ presented in \cite{CMPZ} and \cite{CsMZ2018}.
	\end{itemize}
\end{enumerate}
Very recently, new MRD-codes of minimum distance $d=n-2$ and $n\in\{7,8\}$ have been constructed in \cite{CsMPYarxive}.


For the relationship between MRD-codes and other geometric objects such as linear sets and Segre varieties, we refer to \cite{Lu2017} and also to \cite{ShVdV}.

\bigskip

Since the metric space $\F_q^{n\times n}$ is isomorphic to the metric space ${\rm End}_{\F_q}(\F_{q^n})$ with rank distance defined as $d(f,g):={\rm rk}(f-g)$, taking into account the previous algebra isomorphism between ${\rm End}_{\F_q}(\F_{q^n})$ and $\cL_{n,q}$, it is clear that each $\F_q$-linear RD-code $\cC$ can be regarded as an $\F_q$-vector subspace of $\cL_{n,q}$. Hence, in terms of linearized polynomial, an RD-code of  $\F_q^{n\times n}$, with minimum distance $d$, is an $\F_q$-subspace of $\cL_{n,q}$ and $d:=\min \{d(f,g)\colon f,g\in \cC, f\neq g\}$. Also two given $\F_q$-linear MRD-codes $\cC_1$ and $\cC_2$ are equivalent if and only if there exist $\varphi_1$, $\varphi_2\in \cL_{n,q}$ permuting $\F_{q^n}$ and $\rho\in \mathrm{Aut}(\F_q)$ such that
\[ \varphi_1\circ f^\rho \circ \varphi_2 \in \cC_2 \text{ for all }f\in \cC_1,\]
where $\circ$ stands for the composition of maps and $f^\rho(x)= \sum a_i^\rho x^{q^i}$ for $f(x)=\sum a_i x^{q^i}$.
For a rank distance code $\cC$ given by a set of linearized polynomials, its left and right idealisers can be written as
\[L(\cC)= \{ \varphi \in \cL_{n,q}\colon \varphi \circ f \in \cC \text{ for all }f\in \cC \},\]
and
\[R(\cC)= \{ \varphi \in \cL_{n,q}\colon f \circ \varphi \in \cC \text{ for all }f\in \cC \}.\]
respectively.

Consider the non-degenerate symmetric bilinear form of $\F_{q^n}$ over $\F_q$ defined by $<x,y>= \Tr_{q^n/q}(xy)$,
for each $x,y \in \F_{q^n}$. Then the \emph{adjoint} $\hat{f}$ of the linearized polynomial $f(x)=\sum_{i=0}^{n-1} a_ix^{q^i} \in \mathcal{L}_{n,q}$ with respect to the bilinear form $<,>$ is $\hat{f}(x)=\sum_{i=0}^{n-1} a_i^{q^{n-i}}x^{q^{n-i}}$.
We will refer to $\hat{f}$ simply as the adjoint of $f$, omitting the bilinear form involved.
Hence, we may define the adjoint of a rank distance code $\cC$ given by a set of linearized polynomials as follows
$\cC^\top:=\{\hat{f} \colon f \in \cC\}$.

In \cite{CsMZ2018}, the authors proved that the set $\cC=\la x,x^q+x^{q^3}+cx^{q^5}\ra_{\F_{q^6}}$, $q$ odd, $c^2+c=1$, $q\equiv 0,\pm 1 \pmod 5$ is an $\F_q$-linear MRD-code of $\cL_{6,q}$ of dimension 12, minimum distance 5 and left idealiser isomorphic to $\F_{q^6}$. The right idealiser of $\cC$ is isomorphic to $\F_{q^2}$ (\cite[Appendix B]{PhDthesis}).
In this paper we further investigate the set $\cC$ with the same assumption $c^2+c=1$ and for each value of $q$ (odd and even), obtaining the following result.

\begin{theorem}\label{main}
The set of $q$-polynomials of $\cL_{6,q}$
\[\cC=\la x,x^q+x^{q^3}+cx^{q^5}\ra_{\F_{q^6}},\]
with $c^2+c=1$, is an $\F_q$-linear MRD-code of $\cL_{6,q}$ with dimension 12, minimum distance 5, left idealiser isomorphic to $\F_{q^6}$ and right idealiser isomorphic to $\F_{q^2}$, if and only if $q$ is odd. Moreover,  when $q$ is odd and $q\equiv \pm 2\pmod  5$, $\cC$ is not equivalent to the previously known MRD-codes.
\end{theorem}

Since both the adjoint and the Delsarte dual operations preserve the equivalence of MRD-codes, we have also that the MRD-codes presented in Theorem \ref{main} are not equivalent neither to the adjoint nor to the Delsarte dual of any previously known MRD-code.

\section{$\F_q$-linear MRD-codes and maximum scattered $\F_q$-subspaces}

An $\F_q$-subspace $U$ of rank $n$ of a $2$-dimensional $\F_{q^n}$-space $V$ is {\em maximum scattered} if it defines a scattered $\F_q$-linear set of the projective line $\PG(V,\F_{q^n})$, i.e. $\dim_{\F_q}(U\cap\la{\mathbf v}\ra_{\F_{q^n}})\leq 1$ for each ${\mathbf v}\in V\setminus\{{\mathbf 0}\}$.
Let $V=\F_{q^n}\times\F_{q^n}$, up to the action of the group $\GL(2,q^n)$, an $\F_q$-subspace $U$ of $V$ of rank $n$ can be written as $U=U_f=\{(x,f(x))\colon x\in\F_{q^n}\}$, for some $f\in\cL_{n,q}$.
\bigskip

Sheekey in \cite{Sh} made a breakthrough in the construction of new linear MRD-codes using linearized polynomials (see also \cite{LTZ}).

In \cite{Sh}, the author proved the following result (which have been generalized in \cite[Section 2.7]{Lu2017} and \cite{ShVdV}, see also \cite[Result 4.7]{CsMPZ2019}).

\begin{result}\label{result:MRDLS}	
$\cC$ is an $\F_q$-linear MRD-code of $\cL_{n,q}$ with minimum distance $n-1$ and with left-idealiser isomorphic to $\F_{q^n}$ if and only if up to equivalence
\[\cC=\la x,f(x)\ra_{\F_{q^n}}\]
for some $f \in \cL_{n,q}$	and the $\F_q$-subspace
\[U_{\cC}=\{(x,f(x)) \colon x\in \F_{q^n}\}\]
is a maximum scattered $\F_q$-subspace of $\F_{q^n}\times\F_{q^n}$.

Also, two $\F_q$-linear MRD-codes $\cC$ and $\cC'$ of $\mathcal{L}_{n,q}$, with minimum distance $n-1$ and with left idealisers isomorphic to $\F_{q^n}$, are equivalent if and only if $U_\cC$ and $U_{\cC'}$ are $\Gamma {\mathrm L}(2,q^n)$-equivalent.

\end{result}

So far, the known non-equivalent (under $\Gamma\mathrm{L}(2,q^n)$) maximum scattered $\F_q$-subspaces, yielding to the known non-equivalent $\F_q$-linear MRD-codes with left idealiser isomorphic to $\F_{q^n}$, are
\begin{enumerate}
\item $U^{1,n}_s:= \{(x,x^{q^s}) \colon x\in \F_{q^n}\}$, $1\leq s\leq n-1$, $\gcd(s,n)=1$, see \cite{BL2000,CSZ2016};
\item $U^{2,n}_{s,\delta}:= \{(x,\delta x^{q^s} + x^{q^{n-s}})\colon x\in \F_{q^n}\}$, $n\geq 4$, $\N_{q^n/q}(\delta)\notin \{0,1\}$ \footnote{This condition  implies $q\neq 2$.}, $\gcd(s,n)=1$, see \cite{LP2001} for $s=1$, \cite{Sh,LTZ} for $s\neq 1$;
\item $U^{3,n}_{s,\delta}:= \{(x,\delta x^{q^s}+x^{q^{s+n/2}})\colon x\in \F_{q^{n}}\}$, $n\in \{6,8\}$, $\gcd(s,n/2)=1$, $\N_{q^n/q^{n/2}}(\delta) \notin \{0,1\}$, for the precise conditions on $\delta$ and $q$ see \cite[Theorems 7.1 and 7.2]{CMPZ} \footnote{Also here $q>2$, otherwise $L^{3,n}_{s,\delta}$ is not scattered.};
\item $U^{4}_{c}:=\{(x, x^q+x^{q^3}+c x^{q^5}) \colon x \in \F_{q^6}\}$, $q$ odd, $c^2+c=1$, $q\equiv 0,\pm 1 \pmod 5$, see \cite{CsMZ2018}.
\end{enumerate}

In this paper, we further investigate the $\F_q$-subspaces $U_f$ arising from the trinomial
\[f(x)=x^q+x^{q^3}+c x^{q^5}\in\F_{q^6}[x],\]
with the same assumption $c^2+c=1$ and for each value of $q$ (odd and even), showing that the $\F_q$-subspace $U_f$ of $\F_{q^n}\times\F_{q^n}$ is maximum scattered also for $q$ odd and $q\equiv \pm 2\pmod 5$, whereas it is not scattered for $q$ even.

To do this, as we will see in Section \ref{sec:proof}, studying the Delsarte dual of the code arising from $U_f$ was successful.

\section{The Delsarte dual of an RD-code}

In terms of linearized polynomials, the Delsarte dual of a rank distance code $\cC$ of $\cL_{n,q}$ can be interpreted as follows
\[ \cC^\perp=\{f \in \cL_{n,q} \colon b(f,g)=0 \,\,\, \forall g \in \cC \}, \]
where $b(f,g)=\mathrm{Tr}_{q^n/q}\left(\sum_{i=0}^{n-1} a_ib_i\right)$ for $f(x)=\sum_{i=0}^{n-1} a_ix^{q^i}$, $g(x)=\sum_{i=0}^{n-1} b_i x^{q^i} \in \F_{q^n}[x]$ and $\mathrm{Tr}_{q^n/q}$ denotes the trace function from $\F_{q^n}$ over $\F_q$.

The following result has been proved in \cite{Delsarte}.

\begin{result}\label{result:dual}
Let $\cC$ be an $\F_q$-linear RD-code of $\mathcal{L}_{n,q}$.
Then $\cC$ is an  $\F_q$-linear MRD-code if and only if $\cC^\perp$ is an $\F_q$-linear MRD-code.
\end{result}

\bigskip

Let us consider the set of $\cL_{6,q}$
\[\cC=\la x,x^q+x^{q^3}+cx^{q^5}\ra_{\F_{q^6}},\]
with $c^2+c=1$. By \cite[Theorem 5.1 and Proposition 6.1]{CsMZ2018}, $\cC$ is an $\F_q$-linear MRD-code of $\cL_{6,q}$ with dimension 12, minimum distance $5$, left idealiser isomorphic to $\F_{q^6}$ and right idealiser isomorphic to $\F_{q^2}$ when $q$ is odd and $q\equiv 0,\pm 1 \pmod 5$.

In order to investigate the remaining cases, by Result \ref{result:dual}, we can consider the Delsarte dual RD-code of $\cC$, which is equivalent to
\[ \D=\langle x^q,x^{q^3},-x+x^{q^2}, c^q x - x^{q^4} \rangle_{\F_{q^6}}.\]
Our aim is now to establish under which conditions the RD-code
\[ \D=\langle x^q,x^{q^3},-x+x^{q^2}, c x - x^{q^4} \rangle_{\F_{q^6}} \]
is an MRD-code\footnote{We write $c$ instead of $c^q$, since $c^q$ satisfies $x^2+x=1$.}.

\medskip

The $\F_q$-linear RD-code $\D$ is an MRD if and only if for each nonzero element $f \in \D$ we get $\dim_{\F_q} \ker f \leq 3$.
Since the maximum $q$-degree of the polynomials in $\D$ is $4$ it suffices that do not exist $\alpha,\beta$ and $\gamma$ in $\F_{q^6}$ such that the kernel of
\[ f(x)=\alpha x^q+\beta x^{q^3}+\gamma(-x+x^{q^2})+c x-x^{q^4}= \]
\[ = (-\gamma+c )x+\alpha x^q+ \gamma x^{q^2} +\beta x^{q^3}-x^{q^4} \]
has dimension $4$. Taking into account the characterization of maximum kernel $q$-polynomials when $k=4$ and $n=6$ (cf. \cite[Section 3.4]{CsMPZ2018}) we have the following result.
\begin{proposition}\label{main:prop}
The set of $q$-polynomials
\[\cC=\la x,x^q+x^{q^3}+cx^{q^5}\ra_{\F_{q^6}},\]
with $c^2+c=1$ is an $\F_q$-linear MRD-code of $\cL_{6,q}$ with dimension 12, minimum distance 5 and left idealiser isomorphic to $\F_{q^6}$, if and only if the system
\begin{small}
\begin{equation}\label{eq:rel}
\left\{ \begin{array}{llllllr}
\alpha \neq 0\\
(-\gamma+c)^{\frac{q^6-1}{q-1}}=1\\
(-\gamma +c)[-(-\gamma+c)^{q^4+q^2}+\beta^{q^5+q^4}(-\gamma +c)^{q^4+q^3+q^2}+\beta^{q^2+q}]=1\\
\alpha=-(-\gamma+c)^{q+1}\beta^{q^2}\\
\gamma = -(-\gamma+c)^{q^2+1}+\beta^{q^3+q^2}(-\gamma +c)^{q^2+q+1}\\
\beta= (-\gamma +c)^{q^3+q^2+1} \beta^{q^4}+ \beta^{q^2}(-\gamma+c)^{q^3+q+1}-\beta^{q^4+q^3+q^2}(-\gamma+c)^{q^3+q^2+q+1}
\end{array} \right.
\end{equation}
\end{small}
has no solutions $\alpha,\beta$ and $\gamma$ in $\F_{q^6}$.
\end{proposition}

In the next section we will study System \eqref{eq:rel} when $q$ is odd, $q\equiv \pm 2\pmod 5$ and when $q$ is even separately.

\section{Proof of Theorem \ref{main}}\label{sec:proof}
In this section we prove the main theorem of the paper, showing that System \eqref{eq:rel} has solutions in $\F_{q^6}$ if and only if $q$ is even. By Result \ref{result:dual}, taking \cite[Theorem 5.1 and Proposition 6.1]{CsMZ2018} into account, System \eqref{eq:rel} has no solutions $\alpha,\beta$ and $\gamma$ in $\F_{q^6}$ when $q\equiv0,\pm1\pmod{5}$. Hence, we have to investigate the remaining cases. The resultants presented in this section have been obtained by using the software package MAGMA \cite{magma}.

\subsection{The $q$ odd case, $q\equiv \pm 2\pmod 5$}
From \cite[Section 1.5 (xiv)]{Hir} it follows that $c \in \F_{q^2}\setminus\F_q$, and so $c$ and $c^q$ are the two distinct roots of $x^2+x-1$. Also $c^{q+1}=c+c^q=-1$.


Our aim now is to show that the system
\begin{equation}\label{eq:rel2}
\left\{ \begin{array}{llllllr}
(-\gamma +c)[-(-\gamma+c)^{q^4+q^2}+\beta^{q^5+q^4}(-\gamma +c)^{q^4+q^3+q^2}+\beta^{q^2+q}]=1\\
\gamma = -(-\gamma+c)^{q^2+1}+\beta^{q^3+q^2}(-\gamma +c)^{q^2+q+1}
\end{array} \right.
\end{equation}
has no solutions in the variables $\gamma$ and $\beta$ over $\mathbb{F}_{q^6}$ and as a consequence System \eqref{eq:rel} does not have solutions.
It is straightforward to see that the previous system admits solutions if and only if the following system
\begin{equation}\label{eq:rel3}
\left\{ \begin{array}{llllllr}
\gamma=-\gamma^{q^3}(-\gamma+c)^{q^2+q+1}\\
\left(\frac{1}{-\gamma+c}-\gamma^{q^2}\right)^{\frac{q^6-1}{q+1}}=1
\end{array} \right.
\end{equation}
admits $\F_{q^6}$-rational solutions in the variable $\gamma$.
Clearly, System \eqref{eq:rel2} may be written as
\begin{equation}\label{eq:rel4}
\left\{ \begin{array}{llllllr}
\left(-\gamma^{q-1}(-\gamma+c)\right)^{q^2+q+1}=1\\
\left(\frac{1}{-\gamma+c}-\gamma^{q^2}\right)^{\frac{q^6-1}{q+1}}=1.
\end{array} \right.
\end{equation}

\medskip

The following prelimiary result holds.

\begin{proposition}\label{formsolution}
If $x$ is a solution of System \eqref{eq:rel4}, then
\[ x=\frac{2}{\lambda^{q^2+q}(c+2)-\lambda^{q^2}c+c}, \]
for some $\lambda \in \F_{q^3}^*$ such that $\lambda^{q^2+q+1}=1$.
\end{proposition}
\begin{proof}
By the second equation of \eqref{eq:rel4} it follows that $x\ne 0$ and
\[ \frac{1}{-x+c}-x^{q^2} \]
is a $(q+1)$-th power in $\F_{q^6}^*$. Therefore, there exists $y \in \F_{q^6}^*$ such that
\begin{equation}\label{eq:y1}
y^{q+1}= \frac{1}{-x+c}-x^{q^2},
\end{equation}
and hence
\[ -x+c=\frac{1}{y^{q+1}+x^{q^2}}. \]
The first equation of \eqref{eq:rel4} reads
\[ \left( -\frac{x^{q-1}}{y^{q+1}+x^{q^2}} \right)^{q^2+q+1}=1. \]
Hence, this is equivalent to the existence of $\lambda \in \F_{q^3}^*$ with $\lambda^{q^2+q+1}=1$ and
\begin{equation}\label{eq:y2}
-\lambda x^{q-1}=y^{q+1}+x^{q^2}.
\end{equation}
By Equations \eqref{eq:y1} and \eqref{eq:y2} we have
\[ -x^{q^2}-\lambda x^{q-1}=\frac{1}{-x+c}-x^{q^2}, \]
and hence
\[ \frac{1}\lambda \left( \frac{1}x \right)^q+c \frac{1}x-1=0.\]
Denoting $T:=\frac{1}x$, the above equation becomes
\begin{equation}\label{eq:T}
T^q+c\lambda T-\lambda=0.
\end{equation}
By \cite[Theorem 1.22]{Hir}, since
\[(-\lambda c)^{q^5+q^4+q^3+q^2+q+1}=\lambda^{(q^2+q+1)(q^3+1)}c^{3(q+1)}=-1 \neq 1, \]
Equation \eqref{eq:T} has one solution and it is
\begin{eqnarray*}
\bar T:=&\lambda(-\lambda c)^{q^5+q^4+q^3+q^2+q}+\lambda^q(-\lambda c)^{q^5+q^4+q^3+q^2}+\quad\quad\quad\\&+\lambda^{q^2}(-\lambda c)^{q^5+q^4+q^3}+\lambda^{q^3}(-\lambda c)^{q^5+q^4}\lambda^{q^4}(-\lambda c)^{q^5}+\lambda^{q^5}.
\end{eqnarray*}
Since  $\lambda \in \F_{q^3}^*$ and $c^2+c-1=0$ we have
\[ \bar T=\frac{c+\lambda^{q^2+q}(c+2)-c\lambda^{q^2}}2\]
and the assertion follows.
\end{proof}

By way of contradiction, suppose that System \eqref{eq:rel3} admits at least one solution in $\gamma$, say $x$. Hence $x$ is as in Proposition \ref{formsolution} and the variables of System \eqref{eq:rel4} are $\lambda$ and $c$. In particular the first equation becomes $\lambda^{q^2+q+1}=1$. Our aim is to show that a solution $\lambda$ of the new system obtained in this way does not exists.

Looking at each $q$-power of $\lambda$ as a distinct variable, we define
\[L:=\lambda, M:=\lambda^q, N:=\lambda^{q^2}, C:=c\] and consider the Frobenius images $\lambda^{q^i}$, with $i=0,1,2$ as variables in System \eqref{eq:rel4}. Hence,  we have a weaker system, say $\Sigma$, in the variables $L,M,N,C$.
We want to show that System $\Sigma$ has no solutions over $\mathbb{F}_{q^6}$. Denote by $EQ1$ and by $NEQ2$ the first equation and the numerator of the system $\Sigma$, respectively. Hence
\[EQ1:=LMN-1=0.\]
%
We have that
\begin{eqnarray*}
\mathrm{Resultant}(\mathrm{Resultant}(NEQ2,C^2+C-1,C),EQ1,L)=\\ =2^{14} N^8 \cdot M^{12}\cdot \mathrm{COND1}^2\cdot \mathrm{COND2}^2\cdot \mathrm{COND3}^2,
\end{eqnarray*}
where
\[ \mathrm{COND1}:=M^2N^2 - 2 MN^2 - 4MN + N^2 + 4N - 1, \]
\[ \mathrm{COND2}:=M^2N^2 + 4MN^2 - 4MN - N^2 + 2N - 1, \]
and
\[ \mathrm{COND3}:=M^2N^2 + 4MN^2 - 2MN - N^2 - 4N + 1. \]
Hence, three cases occur.

\begin{itemize}
  \item $\mathrm{COND1}=0$. 
  
Let $Z:=NM-N$, we get
\[ Z^2-4Z-1=0, \]
which implies $Z \in \F_q$ since $\lambda \in \F_{q^3}$.
Therefore $Z-Z^q=0$ and hence the following two resultants should be zero
\[ R1:=\mathrm{Resultant}(Z-Z^q,EQ1,N)=0 \]
and
\[ R2:=\mathrm{Resultant}(\mathrm{COND1},EQ1,N)=0.\]
Also,
\[\textrm{Resultant}(R1,R2,M)=4 L^2 (L^2 - L - 1)=0,\]
i.e. $\lambda \in \F_{q^2}\cap\F_{q^3}=\F_{q}$, which implies $\lambda^3=1$. This means that either $\lambda=1$ or $q\equiv 1 \pmod 3$ and $\lambda^2+\lambda+1=0$. But both cases contradict condition $\lambda^2-\lambda-1=0$.

  \item $\mathrm{COND2}=0$. 
 
In this case the following resultants should be zero
\[ Q1:= \mathrm{Resultant}(\mathrm{COND2},EQ1,N)=0 \]
and
\[ Q2:= \mathrm{Resultant}(\mathrm{COND2}^q,EQ1,N)=0. \]
This implies that
\[\mathrm{Resultant}(Q1,Q2,Lq)=2^4 L^6 (L^2 + 5L - 5) =0.\]
Again $\lambda\in\F_q$ and we get a contradiction when $q\equiv\pm 2 \pmod 5$.

  \item $\mathrm{COND3}=0$. 

Then
\[ S1:=\mathrm{Resultant}(\mathrm{COND3},EQ1,N)=0 \]
and
\[ S2:=\mathrm{Resultant}(\mathrm{COND3},EQ1,N)=0, \]
and then
\[ \mathrm{Resultant}(S1,S2,Lq)= 2^4 L^4(5L^2 - 5L + 1)=0, \]
again a contradiction.
\end{itemize}

The proof is now complete.
%

\subsection{The $q$ even case}

Differently from what happens in the case $q$ odd, we want to show that, when $q$ is even, System \eqref{eq:rel} admits at least a solution of type $(\alpha,\beta,0) \in \mathbb{F}_{q^6}^3$, with $\alpha$ and $\beta$ not zero. Indeed, substituting the value $\gamma=0$ in  System \eqref{eq:rel} we get

\begin{small}
\begin{equation}\label{eq:rel1.1}
\left\{ \begin{array}{llllllr}
\alpha \neq 0\\
c^{\frac{q^6-1}{q-1}}=1\\
c[c^{q^4+q^2}+\beta^{q^5+q^4}c^{q^4+q^3+q^2}+\beta^{q^2+q}]=1\\
\alpha=c^{q+1}\beta^{q^2}\\
0 = c^{q^2+1}+\beta^{q^3+q^2}c^{q^2+q+1}\\
\beta= c^{q^3+q^2+1} \beta^{q^4}+ \beta^{q^2}c^{q^3+q+1}+\beta^{q^4+q^3+q^2}c^{q^3+q^2+q+1}
\end{array} \right.
\end{equation}
\end{small}

Note that the conditions on $\alpha$ will be automatically satisfied once we define $\alpha:=c^{q+1}\beta^{q^2}$ forcing $\beta \ne 0$. Also, the second equation is trivially satisfied since $c \in \mathbb{F}_4^*$ and $c \in \mathbb{F}_{q^2}$.
Hence System \eqref{eq:rel1.1} has solutions if and only if the following system admits solutions

\begin{small}
\begin{equation}\label{eq:rel1.2}
\left\{ \begin{array}{llllllr}
c[c^{q^4+q^2}+\beta^{q^5+q^4}c^{q^4+q^3+q^2}+\beta^{q^2+q}]=1\\
1 =\beta^{q^3+q^2}c^q\\
\beta= c^{q+2} \beta^{q^4}+ \beta^{q^2}c^{2q+1}+\beta^{q^4+q^3+q^2}c^{2(q+1)}
\end{array} \right.
\end{equation}
\end{small}

Since $c^2+c+1=0$, then  $c\in\F_{q^2}$ and $c^3=1$. Then there exists $\beta\in\F_{q^6}$ such that $\beta^{q+1}=1/c^q$. Hence $\beta^{q^3+q^2}=(1/c^q)^{q^2}=1/c^q$ and the second equation of \eqref{eq:rel1.2} is satisfied.
Also, the first equation of System \eqref{eq:rel1.2} reads
$$1=c\bigg[c^2+\frac{1}{c^q}c^{q+2}+\frac{1}{c}\bigg],$$
and hence it is fulfilled. At this point the third equation of \eqref{eq:rel1.2} becomes
$$1= c^{q+2}\beta^{q^4-1}+ c^{2q+1}\beta^{q^2-1}+\beta^{q^4+q^3+q^2-1}c^{2(q+1)},$$
and using that $\beta^{q+1}=1/c^q$ and $c^3=1$, we get that it is satisfied.

\subsection{The right idealiser of RD-codes of Theorem \ref{main}}
Following the computations in \cite[Appendix B]{PhDthesis}), we show that the right idealiser of the RD-codes presented in Theorem \ref{main} is isomorphic to $\F_{q^2}$.
Indeed, let $\varphi(x)$ be an element of $R(\cC)$. Since $\cC$ contains the identity map $\varphi(x) \in \cC$ and hence there exist $\alpha,\beta \in \F_{q^6}$ such that $\varphi(x)=\alpha x+ \beta x^q+\beta x^{q^3}+\beta c x^{q^5}$.
Also,
\[ (x^q+x^{q^3}+c x^{q^5})\circ\varphi(x)=\varphi(x)^q+\varphi(x)^{q^3}+c \varphi(x)^{q^5}\in \cC \]
implying the existence of $a,b \in \F_{q^6}$ such that
\[ \alpha^qx^q+\beta^q(x^{q^2}+x^{q^4}+c^q x)+\alpha^{q^3}x^{q^3}+\beta^{q^3}(x^{q^4}+x+c^{q^3}x^{q^2}) \]
\[+c(\alpha^{q^5}x^{q^5}+\beta^{q^5}(x+x^{q^2}+c^{q^5} x^{q^4})) =ax+b(x^q+x^{q^3}+c x^{q^5}), \]
which is a polynomial identity in $x$.
By comparing the coefficients of terms of degree $q$ and $q^3$ we get $\alpha \in \F_{q^2}$, and by comparing the coefficients of the terms of degree $q^2$ and $q^4$, taking into account that $c\in\F_{q^2}$, we get
\[\beta^q+c^q\beta^{q^3}+c\beta^{q^5}=0 \mbox{ and } \beta^q+\beta^{q^3}+c^{q+1}\beta^{q^5}=0.\]
Subtracting the second equation to the first one, we get $(c^q-1)(\beta^{q^3}-c\beta^{q^5})=0$. Since $c\neq 1$, then $\beta^{q^4}=c^q\beta$ and this equation admits a nonzero solution $\beta\in\F_{q^6}$ if and only if $c^3=1$, contradicting the condition $c^2+c-1=0$.


\subsection{The equivalence issue}
We want to finish this part of the paper showing that the $\F_q$-linear MRD-codes of $\cL_{6,q}$ defined in Theorem \ref{main} are not equivalent to the previously known MRD-codes.

%

From \cite[Section 6]{CMPZ} and \cite[Theorem 6.1]{CsMZ2018}, the previously known $\F_q$-linear MRD-codes of $\cL_{6,q}$ with dimension 12, minimum distance 5 and left idealiser isomorphic to $\F_{q^6}$, up to equivalence, arise from one of the following maximum scattered subspaces of $\F_{q^{6}}\times\F_{q^{6}}$:
$U^{1,6}_{s}$, $U^{2,6}_{s,\delta}$, $U^{3,6}_{s,\delta}$ and $U_c^4$.  Also, from Result \ref{result:MRDLS}, two $\F_q$-linear MRD-codes $\cC$ and $\cC'$ of $\mathcal{L}_{6,q}$, with minimum distance $5$ and with left-idealisers isomorphic to $\F_{q^6}$, are equivalent if and only if $U_\cC$ and $U_{\cC'}$ are $\Gamma {\mathrm L}(2,q^6)$-equivalent.

The stabilisers of the $\F_q$-subspaces above in the group $\GL(2,q^6)$ were determined in \cite[Sections 5 and 6]{CMPZ} and in \cite[Proposition 5.2]{CsMZ2018}. They have the following orders:
\begin{enumerate}
  \item for $U^{1,6}_s$ we have a group of order $q^6-1$,
  \item for $U^{2,6}_{s,\delta}$ and $U^4_c$ we have a group of order $q^2-1$,
  \item for $U^{3,6}_{s,\delta}$ we have a group of order $q^3-1$.
\end{enumerate}

Also, since the $\Gamma\mathrm{L}(2,q^6)$-equivalence preserves the order of such stabilisers and since the results of \cite[Propositions 5.2 and 5.3]{CsMZ2018} do not depend on the congruence of $q$ odd, using the same arguments we prove the last part of Theorem \ref{main}.

\section{New maximum scattered $\F_q$-linear sets of $\PG(1,q^6)$}

A point set $L$ of a line  $\Lambda=\PG(W,\F_{q^n})\allowbreak=\PG(1,q^n)$ is said to be an \emph{$\F_q$-linear set} of $\Lambda$ of rank $n$ if it is
defined by the non-zero vectors of an $n$-dimensional $\F_q$-vector subspace $U$ of the two-dimensional $\F_{q^n}$-vector space $W$, i.e.
\[L=L_U:=\{\la {\bf u} \ra_{\mathbb{F}_{q^n}} \colon {\bf u}\in U\setminus \{{\bf 0} \}\}.\]
One of the most natural questions about linear sets is their equivalence. Two linear sets $L_U$ and $L_V$ of $\PG(1,q^n)$ are said to be \emph{$\mathrm{P\Gamma L}$-equivalent} (or simply \emph{equivalent}) if there is an element in $\mathrm{P\Gamma L}(2,q^n)$ mapping $L_U$ to $L_V$. In the applications it is crucial to have methods to decide whether two linear sets are equivalent or not. This can be a difficult problem and some results in this direction can be found in \cite{CSZ2015, CMP}.

Linear sets of rank $n$ of $\PG(1,q^n)$ have size at most $(q^n-1)/(q-1)$.
A linear set $L_U$ of rank $n$ whose size achieves this bound is called \emph{maximum scattered}.
For applications of these objects we refer to \cite{OP2010} and \cite{Lavrauw}.

To make notation easier, by  $L^{i,n}_{s}$, $L^{i,n}_{s,\delta}$ and $L^4_{c}$  we will denote the $\F_q$-linear set defined by $U^{i,n}_{s}$, $U^{i,n}_{s,\delta}$ and $U^{4}_{c}$, respectively. The $\F_q$-linear sets $\mathrm{P\Gamma L}(2,q^n)$-equivalent to $L^{1,n}_s$ are called \emph{of pseudoregulus type}. It is easy to see that $L^{1,n}_1=L^{1,n}_s$ for any $s$ with $\gcd(s,n)=1$ and that $U^{2,n}_{s,\delta}$ is $\GL(2,q^n)$-equivalent to $U^{2,n}_{n-s,\delta^{-1}}$.

In \cite[Theorem 3]{LP2001} Lunardon and Polverino proved that $L^{2,n}_{1,\delta}$ and $L^{1,n}_1$ are not $\mathrm{P}\Gamma \mathrm{L}(2,q^n)$-equivalent when $q>3$, $n\geq 4$. For $n=5$, in \cite{CMP2019} it is proved that  $L^{2,5}_{2,\delta}$ is $\mathrm{P}\Gamma\mathrm{L}(2,q^5)$-equivalent neither to $L^{2,5}_{1,\delta'}$ nor to $L^{1,5}_1$.

In \cite[Theorem 4.4]{CsMZ2018}, the authors proved that for $n=6,8$ the linear sets $L^{1,n}_1$, $L^{2,n}_{s,\delta}$ and $L^{3,n}_{s',\delta'}$ are pairwise non-equivalent for any choice of $s,s',\delta, \delta'$. Also in \cite[Theorem 5.4]{CsMZ2018} it has been proved that the linear set $L^4_c$ for $q$ odd and $q \equiv 0,\pm 1 \pmod{5}$ is not equivalent to the aforementioned maximum scattered linear sets of $\PG(1,q^6)$. This result has been obtained by \cite[Proposition 5.3]{CsMZ2018}, where the congruences of $q$ odd plays no role. Hence, using the same arguments, we have the following result.

\begin{theorem}\label{thm:linset}
The $\F_q$-linear set $L_c$ of rank 6 of $\PG(1,q^6)$ defined by the $\F_q$-subspace of $\F_{q^6}\times\F_{q^6}$
\[U_c=\{(x,x^q+x^{q^3}+cx^{q^5})\colon x\in\F_{q^6}\},\]
with $c^2+c=1$, is scattered if and only if $q$ is odd. Also, when $q\equiv \pm 2\pmod  5$ $L_c$ is not $\mathrm{P\Gamma L}(2,q^6)$-equivalent to the previously known
maximum scattered $\F_q$-linear sets of $\PG(1,q^6)$.
\end{theorem}

\section*{Final remark}
In this paper we have proved that the RD-code $\cC=\la x,f(x)\ra_{\F_{q^{2n}}}$ of $\cL_{2n,q}$, with $f(x)=x^q+x^{q^3}+cx^{q^5}\in\F_{q^6}$ ($n=3$) and $c^2+c+1=0$, is an MRD-code with dimension 12, minimum distance 5 and left idealiser isomorphic to $\F_{q^6}$ if and only if $q$ is odd. Computational results show that, for suitable choices of $c\in\F_{q^6}\setminus(\F_{q^2}\cup\F_{q^3})$ the previous trinomial produces MRD-codes also when $q\leq 64$ is even.

We strongly believe that the previous MRD-codes belong to a larger class of MRD-codes, arising from polynomials of type $f(x)=x^q+\sum_{i=1}^{n-1}a_{{2i+1}}x^{q^{2i+1}}\in\F_{q^{2n}}[x]$, under suitable assumptions on the coefficients $a_j$'s. In Table 1 we provide some explicit examples. They are the results of our successful searches using the software package MAGMA \cite{magma} for small values of $n$ and $q$. When a parameter $a_i$ appears in a row of the table it means that there exist explicit values of $a_i \in\F_{q^{2n}}$ for which the polynomial $f(x)$ gives rise to an MRD-code. Certainly, a careful study of the corresponding $\F_q$-subspaces of $\F_{q^{2n}}\times\F_{q^{2n}}$ should be undertaken in order to establish whether the MRD-codes are equivalent to the previously known ones. The authors are currently beginning work on these two projects.
%
%
%

\begin{table}[htp]
\[
\begin{array}{ |c|c|c| }
\hline
$n$ & $q$ & $f(x)$ \\
\hline
3 & q\leq 64, \text{even} & x^q+x^{q^3}+a_5x^{q^5} \\ \hline
3 & 3,5& x^q-x^{q^3}+a_5x^{q^5}   \\ \hline
3 & 3,5,7& x^q+a_3 x^{q^3}+a_3^2x^{q^5}   \\ \hline

4 & 3,5& x^q+x^{q^3}+x^{q^5}-x^{q^7}  \\ \hline
4 & 4& x^q+a_3^2x^{q^3}+a_3x^{q^5}+x^{q^7}  \\ \hline
5 & 3 & x^q+a_3x^{q^3}+a_5x^{q^5}+a_7x^{q^7}+a_9x^{q^9}  \\ \hline
\end{array}
\]
\caption{Computational results}
\label{table:cond}
\end{table}
%
%
%

\bigskip

\noindent
Giuseppe Marino\\
Dipartimento di Matematica e Applicazioni “Renato Caccioppoli”\\
Universit\'a degli Studi di Napoli “Federico II”\\
Via Cintia, Monte S.Angelo I-80126 Napoli\\
Italy\\
\emph{giuseppe.marino@unina.it}

\bigskip

\noindent
Maria Montanucci\\
Technical University of Denmark\\
Asmussens All\'e\\
Building 303B, room 150\\
2800 Kgs. Lyngby\\
Denmark\\
\emph{marimo@dtu.dk}

\bigskip

\noindent
Ferdinando Zullo\\
Dipartimento di Matematica e Fisica\\
Universit\`a degli Studi della Campania \emph{Luigi Vanvitelli}\\
Viale lincoln, 5\\
36100 Caserta\\
Italy\\
\emph{ferdinando.zullo@unicampania.it}

\end{document}